\documentclass{amsart}
\usepackage{amssymb,latexsym}

\theoremstyle{plain}
\newtheorem{theorem}{Theorem}

\newtheorem{lemma}{Lemma}
\theoremstyle{definition}

\newtheorem{remark}{Remark}

\usepackage{amssymb,latexsym}
\usepackage{enumerate, framed}

\begin{document}

\title[symmetric tensor rank]{Real and complex rank for real symmetric tensors
with low ranks}
\author{Edoardo Ballico, Alessandra Bernardi}
\address{University of Trento\\
  Dept. of Mathematics\\
38123 Povo (TN), Italy}
\address{University of Torino\\
  Dept. of Mathematics ``Giuseppe Peano''\\
  10123 Torino, Italy}
\email{ballico@science.unitn.it, alessandra.bernardi@unito.it}
\thanks{The authors were partially supported by CIRM of FBK Trento 
(Italy), Project Galaad of INRIA Sophia Antipolis M\'editerran\'ee 
(France), Institut Mittag-Leffler (Sweden), Marie Curie: Promoting science (FP7-PEOPLE-2009-IEF), MIUR and GNSAGA of 
INdAM (Italy).}
\subjclass{14N05; 15A69}
\keywords{symmetric tensor rank; Veronese embedding; real symmetric tensor rank}

\begin{abstract}We study the case of a real homogeneous polynomial $P$ whose  minimal real and complex decompositions in terms of powers of linear forms are different. We prove that, if the sum of the complex and the real ranks of $P$ is at most $ 3\deg(P)-1$, then the difference of the two decompositions is completely determined either on a line or on a conic or two disjoint lines.
\end{abstract}

\maketitle

%\begin{keywords}
%symmetric tensor rank; Veronese embedding; real symmetric tensor rank.
%\end{keywords}

%\begin{AMS}
%14N05; 15A69
%\end{AMS}

\pagestyle{myheadings}
\thispagestyle{plain}

\section*{Introduction}\label{S1}

The problem of decomposing a tensor into a minimal sum of rank-1 terms, is raising interest and attention from many applied areas as signal processing for telecommunications \cite{deLC}, independent component analysis \cite{Comon}, complexity of matrix multiplication \cite{Stra83:laa},  complexity problem of  P versus NP \cite{Vali},  quantum physics \cite{entanglement}, \cite{bi} and phylogenetic  \cite{ar}. 
The particular instance in which the tensor is symmetric and hence representable by a homogeneous polynomial, is one of the most studied and developed ones (cfr. \cite{l} and references therein). In this last case we say that the rank of a homogeneous polynomial $P$ of degree $d$ is the minimum integer $r$ needed to write it as a linear combination of pure powers of linear forms $L_1, \ldots , L_r$:
\begin{equation}\label{F}
P=c_1L_1^d + \cdots +  c_dL_r^d.
\end{equation}

with $c_i\ne 0$. Most of the papers concerning the abstract theory of the symmetric tensor rank, require that the base field is algebraically closed. In this case
we may take $c_i=1$ for all $i$ without loss of generality. However, for the applications, it is very important to consider the case of real polynomials and look at their real decomposition. Namely, one can study separately the case in which the linear forms appearing in (\ref{F}) are complex or reals. In the
real case we may take $c_i=1$ for all $i$ if $d$ is odd, while we take $c_i\in \{-1,1\}$ if $d$ is even. When we look for a minimal complex (resp. real) decomposition as in (\ref{F}) we  say that we are computing the \emph{complex symmetric rank} (resp. \emph{real symmetric rank}) of $P$ and we will indicate it $r_{\mathbb{C}}(P)$ (resp. $r_{\mathbb{R}}(P)$). Obviously
$$r_{\mathbb{C}}(P)\leq r_{\mathbb{R}}(P),$$
and in many cases such an equality is strict.

In \cite{co} P. Comon and G. Ottaviani studied the real case for bivariate symmetric tensors. Even in this case there are many open conjectures and, up to now, few cases are completely settled (\cite{co}, \cite{cr}, \cite{b}, \cite{bl}).

In this paper we want to study the relation among $r_{\mathbb{C}}(P)$ and $r_{\mathbb{R}}(P)$ in the special circumstance in which $r_{\mathbb{C}}(P) < r_{\mathbb{R}}(P)$. In particular we will show that, in a certain range (say, $r_{\mathbb{C}} (P) +r_{\mathbb{R}} (P) \le 3\deg (P)-1$),
all homogeneous polynomials $P$ of that degree with $r_{\mathbb{R}} (P) \ne r_{\mathbb{C}} (P)$ are characterized by the existence of a curve with the property that the sets evincing the real and the complex ranks coincide out of it (see Theorem \ref{i1} for the precise statement). More precisely: let $P\in S^d\mathbb{R}^{m+1}$ be a real homogeneous polynomial of degree $d$ in $m+1$ variables such that $r_{\mathbb{C}}(P) < r_{\mathbb{R}}(P)$ and $r_{\mathbb{C}} (P) +r_{\mathbb{R}} (P) \le 3\deg (P)-1$; therefore its real and complex decomposition are
$$P=a_1L_1^d+ \cdots + a_kL_k^d+c_1M_1^d+ \cdots +c_sM_s^d,$$
$$P=N_1^d+ \cdots + N_h^d+c_1M_1^d+ \cdots +c_sM_s^d$$
respectively, with $k>h$, $a_i\in \{-1,1\}$, $c_j\in \{-1,1\}$, $M_1, \ldots , M_s\in S^1\mathbb{R}^{m+1}$, $h+k>d+2$. Moreover there exists a curve $C' \subset \mathbb {P}^m$ such that $[L_1], \ldots , [L_k], [N_1], \ldots , [N_h]$
``~depends only from the variables of $C'$~'' and $C'$ is either a line or a reduced conic or a disjoint union of two lines.
If $C'$ is a line (item (\ref{a}) in Theorem \ref{i1}) then both the $L_i$'s and the $N_i$'s are linear forms in the same two ``variables". If $C'$ is a conic, then $L_i$'s and $N_i$'s depend on 3 ``variables" and their projectivizations lie on $C'$. See item (c) of Theorem \ref{i1} for the geometric interpretation
of the reduction of $L_1,\dots ,L_k$ and $N_1,\dots ,N_h$ to bivariate forms involved with
$C'$ when $C' =l\sqcup r$ is a disjoint union of two lines $l$ and $r$ (we have two sets of bivariant forms, one for the variables of $l$ and one for the variables of $r$).

\section{Notation and statements}

Before giving the precise statement of Theorem \ref{i1} we need to introduce the main algebraic geometric tools that we will use all along the paper.

Let $\nu _d: \mathbb {P}^m \to \mathbb {P}^N$, $N:= {m+d \choose d}-1$, 
denote the degree $d$ Veronese embedding of $\mathbb {P}^m$ (say, defined over $\mathbb {C})$.
Set $X_{m,d}:= \nu _d(\mathbb {P}^m)$.
For any $P\in \mathbb {P}^N$, the {\emph {symmetric rank}} or {\emph 
{symmetric tensor rank}} or, just, the {\emph {rank}} 
$r_{\mathbb{C}} (P)$ of $P$ is the minimal cardinality of a finite set 
$S\subset \mathbb {P}^m(\mathbb {C})$ such that $P\in \langle \nu _d(S)\rangle$, where
$\langle \ \ \rangle$ denote the linear span (here the linear span is with respect to complex coefficients), and we will say that $S$ \emph{evinces} $r_{\mathbb{C}}(P)$.
Notice that the Veronese embedding $\nu _d$ is defined over $\mathbb {R}$, i.e.
$\nu _d(\mathbb {P}^m(\mathbb {R})) \subset \mathbb {P}^N(\mathbb {R})$. For each $P\in \mathbb {P}^N(\mathbb {R})$ the \emph{real symmetric rank} $r_{\mathbb{R}} (P)$ of $P$ is the minimal cardinality
of a finite set $S\subset \mathbb {P}^m(\mathbb {R})$ such that $P\in \langle \nu _d(S)\rangle _{\mathbb {R}}$, where $\langle \ \ \rangle _{\mathbb {R}}$ means the linear span with real coefficients, and we will say that $S$ evinces $r_{\mathbb{R}}(P)$. The integer $r_{\mathbb{R}} (P)$ is well-defined, because $\nu _d(\mathbb {P}^m(\mathbb {R}))$ spans $\mathbb {P}^N(\mathbb {R})$. 

Let us fix some notation: If $C\subset \mathbb{P}^m$ is either a curve or a subspace and $S\subset \mathbb{P}^m$ is a finite set, we will use the following abbreviations:
$$S_C:=S\cap C,$$
$$S_{\hat{C}}:=S\setminus (S\cap C).$$

\begin{theorem}\label{i1}
Let $P\in \mathbb {P}^N(\mathbb {R})$ be such that $r_{\mathbb{C}} (P)+r_{\mathbb{R}} (P) \le 3d-1$
and $r_{\mathbb{C}} (P) \ne r_{\mathbb{R}} (P)$. Fix any set $S_{\mathbb{C}}\subset \mathbb {P}^m(\mathbb {C})$ and ${S_{\mathbb{R}}}\subset \mathbb {P}^m(\mathbb {R})$  evincing $r_{\mathbb{C}} (P)$ and  $r_{\mathbb{R}} (P)$ respectively. Then
one of the following cases (a), (b), (c) occurs:

\begin{enumerate}[(a)]

\item\label{a} There is a line $l\subset \mathbb {P}^m$ defined
over $\mathbb {R}$ and with the following properties:

\begin{enumerate}[(i)]
	
	\item\label{ai} $S_{\mathbb{C}}$ and $S_{\mathbb{R}}$ coincide out of the line $l$ in a set $S_{\hat{l}}$: $$S_{\mathbb{C}}\setminus S_{\mathbb{C}}\cap l = {S_{\mathbb{R}}}\setminus {S_{\mathbb{R}}}\cap l=:S_{\hat{l}};$$
	
	\item\label{aii} there is a point $P_l\in \langle \nu _d(S_{\mathbb{C},l})\rangle \cap \langle \nu _d(S_{\mathbb{R},l})\rangle$ such that $S_{\mathbb{C},l}$ evinces $r_{\mathbb{C}} (P_l)$ and $S_{\mathbb{R},l}$ evinces $r_{\mathbb{R}} (P_l)$;
	
	\item\label{aiii} $\sharp (S_{\mathbb{C},l}\cup {S_{\mathbb{R},l}}) \ge d+2$ and $\sharp (S_{\mathbb{C},l})< \sharp (S_{\mathbb{R},l})$.

\end{enumerate}

\item\label{b} There is a conic $C\subset \mathbb {P}^m$ defined
over $\mathbb {R}$ and with the following properties:

\begin{enumerate}[(i)]

	\item\label{bi} $S_{\mathbb{C}}$ and $S_{\mathbb{R}}$ coincide out of the conic $C$ in a set  $S_{\hat{C}}$: $$S_{\mathbb{C}}\setminus S_{\mathbb{C},C}= {S_{\mathbb{R}}}\setminus {S_{\mathbb{R}}}\cap C=:S_{\hat{C}};$$
	
	\item\label{bii}  there is a point $P_C\in \langle \nu _d(S_{\mathbb{C},C})\rangle \cap \langle \nu _d(S_{\mathbb{R},C})\rangle$ such that $S_{\mathbb{C},C}$ evinces $r_{\mathbb{C}} (P_C)$ and $S_{\mathbb{R},C}$ evinces $r_{\mathbb{R}} (P_C)$;
	
	\item\label{biii} $\sharp (S_{\mathbb{C},C}\cup {S_{\mathbb{R},C}}) \ge 2d+2$ and $\sharp (S_{\mathbb{C},C})< \sharp (S_{\mathbb{R},C})$;

	\item\label{biv} if $C$ is reducible, say $C = l_1\cup l_2$ with $Q=l_1\cap l_2$, then $\sharp ((S_{\mathbb{C}}\cup {S_{\mathbb{R}}})\cap (l_i\setminus Q)) \ge d+1$ for $i\in \{1,2\}$.

\end{enumerate}

\item\label{c} $m\ge 3$ and there are $2$ disjoint lines $l, r \subset \mathbb {P}^m$ defined over $\mathbb {R}$  with the following properties:

\begin{enumerate}[(i)]
	
	\item\label{ci} $S_{\mathbb{C}}$ and $S_{\mathbb{R}}$ coincide out of the union $\Gamma:=l\cup r$ in a set $S_{\hat{\Gamma}}$:
	$$S_{\mathbb{C}}\setminus S_{\mathbb{C}}\cap (l\cup r) = 
	S_{\mathbb{R}}\setminus S_{\mathbb{R}}\cap (l\cup r):=S_{\hat{\Gamma}};$$
	
	\item\label{cii} $\sharp (S_{\mathbb{C},l}\cup {S_{\mathbb{R},l}}) \ge d+2$ and $\sharp (S_{\mathbb{C},r}\cup {S_{\mathbb{R},r}}) \ge d+2$;

		\item\label{ciii} the set  $\langle  \nu_d(S_{\hat{\Gamma}})\rangle\cap \langle\nu _d(\Gamma)\rangle$
is a single point, $O_{\Gamma}\in \mathbb {P}^N(\mathbb {R})$; \\$S_{\mathbb{C},\Gamma}$ evinces $r_{\mathbb{C}} (O_{\Gamma})$ and ${S_{\mathbb{R},\Gamma}}$ evinces $O_{\Gamma}$; 
		\item\label{civ} the set $\langle \{O_\Gamma\}\cup \nu _d(l)\rangle \cap \langle \nu
_d(l)\rangle$ (resp. $\langle \{O_{\Gamma}\}\cup \nu _d(r)\rangle \cap \langle \nu
_d(r)\rangle$) is formed by a unique point $O_l\in \mathbb {P}^N(\mathbb {R})$ (resp. $O_r\in \mathbb {P}^N(\mathbb {R})$);\\
$ S_{\mathbb{C},l}$ (resp. $S_{\mathbb{C},r}$) evinces $r_{\mathbb{C}} (O_l)$ (resp.  $r_{\mathbb{C}} (O_r)$);\\
${S_{\mathbb{R},l}}$ (resp. $ {S_{\mathbb{R},r}}$) evinces $r_{\mathbb{R}} (O_l)$ (resp $r_{\mathbb{R}} (O_r)$).
 \end{enumerate} \end{enumerate} \end{theorem}

\section{The proof}

\begin{remark}\label{d0} Let $S\subset \mathbb{P}^N(\mathbb{R})$. It will be noteworthy in the sequel  that  $S$ can be used to span both a real  space $\langle S \rangle_{\mathbb{R}}\subset \mathbb{P}^N(\mathbb{R})$ and a complex space $\langle S \rangle_{\mathbb{C}}\subset \mathbb{P}^N(\mathbb{C})$ of the same dimension and $\langle S \rangle_{\mathbb{C}}\cap \mathbb{P}^N(\mathbb{R})=
\langle S \rangle_{\mathbb{R}}$.
In the following we will always use $ \langle \ \ \rangle$
to denote $\langle \ \  \rangle_{\mathbb{C}}$.
\end{remark}

\begin{remark}\label{d1}
Fix $P\in \mathbb {P}^N$ and a finite set $S\subset \mathbb {P}^N$ such that
$S$ evinces $r_{\mathbb{C}} (P)$. Fix
any $E\subsetneq S$. Then the set $\langle \{P\} \cup E \rangle \cap \langle S\setminus E\rangle$ is a single point (call it $P_1$) and $S\setminus
E$ evinces $r_{\mathbb{C}} (P_1)$. Now assume $P\in \mathbb {P}^N(\mathbb {R})$ and $S\subset  \mathbb {P}^N(\mathbb {R})$. Then $P_1\in \mathbb {P}^N(\mathbb {R})$. If  $S$ evinces $r_{\mathbb{R}} (P)$, then $S\setminus E$ evinces $r_{\mathbb{R}} (P_1)$.
\end{remark}

\begin{lemma}\label{c2} Let $C\subset \mathbb {P}^m$ be a reduced curve of degree $t$ with $t=1,2$. Fix finite sets $A,B \subset \mathbb {P}^m$. Fix an integer $d>t$ such
that:
$$h^1(\mathcal {I}_{(A\cup B)_{\hat{C}}}(d-t))=0.$$
Assume the existence of $P\in \langle \nu _d(A)\rangle \cap \langle \nu _d(B)\rangle$ and $P\notin \langle \nu _d(S')\rangle$ for any $ S' \subsetneq A$ and any $S' \subsetneq {B}$. Then
$$A_{\hat{C}}=B_{\hat{C}}.$$
\end{lemma}

\begin{proof} The case $t=1$ is \cite[Lemma 8]{bb1}.
If $t=2$, then either $C$ is a conic or $m\ge 3$ and $C$ is a disjoint union of $2$ lines.
In both cases we have $h^0(\mathcal {I}_C(t)) >0$
and the linear system $\vert \mathcal {I}_C(t)\vert$ has no base points outside $C$. Since
$A\cup B$ is a finite set,
there is $M\in \vert \mathcal {I}_C(t)\vert$ such that $M\cap (A\cup  B) = C\cap (A\cup B)$. Look
at the residual exact sequence (also called the Castelnuovo's exact sequence):
\begin{equation}\label{eqa1}
0 \to \mathcal {I}_{(A\cup B)_{\hat{C}}}(d-t) \to \mathcal {I}_{A\cup B}(d)
\to \mathcal {I}_{(A\cup B)\cap M,M}(d) \to 0
\end{equation}
We can now repeat the same proof of \cite[Lemma 8]{bb1} but starting with (\ref{eqa1}) instead of the exact sequence used there (cfr. first displayed formula in the proof of \cite[Lemma 8]{bb1}).

We will therefore get
$A_{\hat{M}} = B_{\hat{M}}$. Now, since $M\cap (A\cup  B) = C\cap (A\cup B)$, we
are done.
\end{proof}

We are now going to prove Theorem \ref{i1}.

\vspace{0.3cm}

\quad {\emph {Proof of Theorem \ref{i1}}} 

Fix $P\in \mathbb {P}^N(\mathbb {R})$ such that $r_{\mathbb{C}} (P)+r_{\mathbb{R}} (P) \le 3d-1$
and $r_{\mathbb{C}} (P) \ne r_{\mathbb{R}} (P)$. 
\\ Fix any set $S_{\mathbb{C}}\subset \mathbb {P}^m(\mathbb {C})$
evincing $r_{\mathbb{C}} (P)$ and any ${S_{\mathbb{R}}}\subset \mathbb {P}^m(\mathbb {R})$ evincing $r_{\mathbb{R}} (P)$.
\\
By applying \cite{bb}, Lemma 1, we immediately get that
$$h^1(\mathcal {I}_{S_{\mathbb{C}}\cup {S_{\mathbb{R}}}}(d)) >0.$$ 
Since $\sharp (S_{\mathbb{C}})+\sharp ({S_{\mathbb{R}}}) \le 3d-1$, either there is a line $l\subset \mathbb {P}^m$ such that $\sharp (S_{\mathbb{C},l}\cup {S_{\mathbb{R},l}}) \ge d+2$ or there is a conic $C$ such that $\sharp  (S_{\mathbb{C},C}\cup {S_{\mathbb{R},C}}) \ge 2d+2$ (\cite{c2}, Theorem 3.8). We are going to study separately these two cases in items (\ref{1}) and (\ref{2}) below.

\begin{enumerate}

\item\label{1} In this step we assume the existence of a line $l \subset \mathbb {P}^m$ such that 
$$\sharp ( S_{\mathbb{C},l}\cup {S_{\mathbb{R},l}})\ge d+2.$$

\newcounter{mycounter}
\setcounter{mycounter}{\theenumi}
\end{enumerate}

This hypothesis, together with $r_{\mathbb{R}}(P)\neq r_{\mathbb{C}}(P)$,  immediately implies property (\ref{aiii}) of the statement of the theorem.

We are now going to distinguish the case $h^1(\mathcal {I}_{S_{\mathbb{C}}\cup {S_{\mathbb{R}}}}(d-1))=0$ (item (1.1) below) from the case $h^1(\mathcal {I}_{S_{\mathbb{C}}\cup {S_{\mathbb{R}}}}(d-1))>0$ (item (1.2) below).

\vspace{0.1cm}
\quad (1.1)  Assume $h^1(\mathcal {I}_{S_{\mathbb{C}}\cup {S_{\mathbb{R}}}}(d-1))=0$.

First of all, observe that the line $l\subset \mathbb{P}^m$ is well defined over $\mathbb{R}$ since it contains at least 2 points of ${S_{\mathbb{R}}}$ (Remark \ref{d0}). Then, by Lemma \ref{c2}, we have that ${S_{\mathbb{C}}}$ and ${S_{\mathbb{R}}}$ have to coincide out of the line $l$:  
$${S_{\mathbb{C}}}\setminus S_{\mathbb{C},l} = {S_{\mathbb{R}}}\setminus S_{\mathbb{R},l}:=S_{\hat{l}},$$
and this proves  (\ref{ai}) of the statement of the theorem in this case (1.1). 

The fact that $\sharp (S_{\mathbb{C}}) < \sharp ({S_{\mathbb{R}}})$  implies that $\sharp (S_{\mathbb{R},l}) > (d+2)/2$ and $\sharp (S_{\hat{l}}) \le d$, 
hence  
$h^1(\mathcal {I}_{S_{\hat{l}}\cup l}(d))=0.$

Therefore we have that $\dim (\langle \nu _d(S_{\hat{l}} \cup l)\rangle )= \sharp (S_{\hat{l}})+d+1$, $\dim (\langle \nu
_d(S_{\hat{l}})\rangle ) =\sharp (S_{\hat{l}})-1$ and $\dim (\langle \nu _d(l)\rangle )=d+1$, and Grassmann's formula
gives $\langle \nu _d(S_{\hat{l}})\rangle \cap \langle \nu _d(l)\rangle = \emptyset$. \\

Since $P\in \langle \nu _d(S_{\hat{l}}\cup S_{\mathbb{C}})\rangle$ and $S_{\mathbb{C},l}\subset l$, the
set $\langle \nu _d(S_{\hat{l}})\cup \{P\}\rangle \cap \langle \nu _d(l)\rangle$ is a single point, $P_l\in \mathbb {P}^N(\mathbb {R})$. \\

Since $P\in \langle \nu _d(S_{\mathbb{C}})\rangle$ and $P\notin \langle \nu _d(S_{\mathbb{C}}')\rangle$ for any $S'\subsetneq S_{\mathbb{C}}$,
the set $\langle \nu _d(S_{\hat{l}})\cup \{P\}\rangle \cap \langle \nu _d(S_{\mathbb{C},l})\rangle$ is a single point, $P_C$ 
(Remark \ref{d1}). 

Then obviously:
$$P_C=P_l\in \mathbb{P}^N(\mathbb{R}).$$ 
Since $S_{\mathbb{C}}$ evinces $r_{\mathbb{C}} (P)$, then $S_{\mathbb{C},l}$ evinces $P_l$ (Remark \ref{d1}). In the same way we see that  $\langle \nu _d(S_{\hat{l}})\cup \{P\}\rangle \cap \langle \nu _d(S_{\mathbb{R},l})\rangle =\{P_l\}$ and that $S_{\mathbb{R},l}$ evinces $r_{\mathbb{R}} (P_l)$.
%Since $l\subset \mathbb {P}^m$, the set $S_{\mathbb{C},l}$ (resp. $S_{\mathbb{R},l}$) evinces the complex (resp. real) symmetric tensor rank of $P_l$ with respect to the rational normal curve $\nu _d(l)$.
This proves (\ref{aii}) of Theorem \ref{i1} in this case (1.1). 

\vspace{0.1cm}
\quad (1.2) Assume $h^1(\mathcal {I}_{S_{\mathbb{C}}\cup {S_{\mathbb{R}}}}(d-1))>0$.

First of all, observe that there exists a line $r\subset \mathbb {P}^m$
such that $\sharp (r\cap (S_{\mathbb{C}}\cup {S_{\mathbb{R}}})_{\hat{l}})\ge d+1$, because $\sharp (S_{\mathbb{C}}\cup {S_{\mathbb{R}}})_{\hat{l}} \le 3d-1-d-2 \le 2(d-1)+1$. 
%there is a line $r\subset \mathbb {P}^m$ such that $\sharp (r\cap (S_{\mathbb{C}}\cup {S_{\mathbb{R}}})_{\hat{l}})\ge d+1$. 
\\
By the same reason, if we write: $C:=l\cup r$, we get that $\sharp (S_{\mathbb{C}}\cup {S_{\mathbb{R}}})_{\hat{C}}\le 3d-1 -d-2-d-1 \le d-2$ and hence   $h^1(\mathcal {I}_{(S_{\mathbb{C}}\cup {S_{\mathbb{R}}})_{\hat{C}}}(d-2))=0$ 
(e.g. by
\cite{bgi}, Lemma 34, or by \cite{c2}, Theorem 3.8). Lemma \ref{c2}
gives: 
\begin{equation}\label{SChat}
S_{\mathbb{C},{\hat{C}}}= {S_{\mathbb{R},{\hat{C}}}}. 
\end{equation}

-- Assume for the moment $l\cap r \ne \emptyset$. 
In this case, Remark \ref{d1} indicates that we can consider case (\ref{b}) of the statement of the theorem. Therefore (\ref{SChat}) proves (\ref{bi}) in the case that the conic $C$ in (\ref{b}) in the statement of the theorem is reduced. 
%\marginnote{\tiny{[Thm.\ref{i1}, \ref{bi} (tbc)]}} 
Moreover, condition (\ref{biv}) is satisfied, because $\sharp (r\cap (S_{\mathbb{C}}\cup {S_{\mathbb{R}}})_{\hat{l}}) \ge d+1$. 
%\marginnote{\tiny{[Thm.\ref{i1}, \ref{biv}]}}

-- Now assume $l\cap r = \emptyset$.  We will check that we are in case (\ref{a}) with respect to the line $l$ if $\sharp (S_{\mathbb{C},r}\cup {S_{\mathbb{R},r}}) =d+1$, while we are in case (\ref{c}) with respect to the lines
$l$ and $r$ if $\sharp  (S_{\mathbb{C},r}\cup {S_{\mathbb{R},r}}) \ge d+2$, and the case $\sharp (S_{\mathbb{C},l}\cup {S_{\mathbb{R},l}}) = \sharp (S_{\mathbb{C},r}\cup {S_{\mathbb{R},r}})
= d+1$ cannot occur.

Set $\Gamma:=  l\cup r$. Since $r\cap l =\emptyset$, we have $\dim \langle \Gamma \rangle =3$ and hence
$m \ge 3$. 

--- Assume for the moment $m\ge 4$.
%Since $M\supset l\cup r$, we have
Hence $\sharp (S_{\mathbb{C}}\cup {S_{\mathbb{R}}})_{\widehat{\langle \Gamma \rangle}} \le d$ and $h^1(\mathcal {I}_{ (S_{\mathbb{C}}\cup {S_{\mathbb{R}}})_{\widehat{\langle \Gamma \rangle}}}(d-1))=0$. Therefore:

\begin{equation}\label{CRGamma}
S_{\mathbb{C}, \widehat{\langle\Gamma\rangle}}= {S_{\mathbb{R}, \widehat{\langle\Gamma\rangle}}}
\end{equation}
and the set $\langle \{P\}\cup \nu _d(S_{\mathbb{C}, \widehat{\langle\Gamma\rangle}})\rangle \cap
\langle \nu _d(\langle \Gamma \rangle)\rangle$ is a single real point:

\begin{equation}\label{Oset}
O:=\langle \{P\}\cup \nu _d(S_{\mathbb{C}, \widehat{\langle\Gamma\rangle}})\rangle \cap
\langle \nu _d(\langle \Gamma \rangle)\rangle \in \mathbb{P}^N(\mathbb{R}),
\end{equation}
 $S_{\mathbb{C},\Gamma}$ evinces $r_{\mathbb{C}} (O)$
and ${S_{\mathbb{R}, \Gamma}}$ evinces $r_{\mathbb{R}} (O)$. Now, (\ref{CRGamma}) implies that if we are either in case (\ref{a}) or in case (\ref{c}) of the theorem, we can simply study what happens at $S_{\mathbb{C}, {\langle\Gamma\rangle}}$ and at $ {S_{\mathbb{R},{\langle\Gamma\rangle}}}$, which means that we can reduce our study to  the case  $m=3$,  since $\langle \Gamma \rangle= \mathbb{P}^3$.

--- Until step (\ref{2}) below, we will assume $m=3$.

The linear system $\vert \mathcal {I}_{\Gamma}(2)\vert$ on $\langle \Gamma \rangle$ has no base points outside $\Gamma$ itself. Since
$S_{\mathbb{C}}\cup {S_{\mathbb{R}}}$ is finite, there is a smooth quadric surface $W$ containing $\Gamma$ such that
$$S_{\mathbb{C},W}\cup {S_{\mathbb{R},W}} = S_{\mathbb{C},\Gamma}\cup {S_{\mathbb{R},\Gamma}}.$$ 
Moreover, such a $W$ can be found among the real smooth quadrics,  since $l$ and $r$ are real lines.\\

Since $\sharp (S_{\mathbb{C},\langle \Gamma \rangle}\cup {S_{\mathbb{R}, \langle \Gamma \rangle}})_{\hat{W}}
\le d-1$, we have $h^1(\mathcal {I}_{(S_{\mathbb{C}}\cup {S_{\mathbb{R}}})_{\hat{W}}}(d-2))=0$. Hence
Lemma \ref{c2} applied to the point $O$ defined in (\ref{Oset}),
gives:
$$(S_{\mathbb{C}, \langle \Gamma \rangle})_{\hat{W}}= (S_{\mathbb{R},\langle \Gamma \rangle})_{\hat{W}},$$
$\langle \{O\}\cup \nu _d(S_{\mathbb{C},\langle \Gamma , \rangle})_{\hat{W}}\rangle \cap \langle \nu _d(W)\rangle$ is a single real point 
\begin{equation}\label{Opset} 
O'=\langle \{O\}\cup \nu _d(S_{\mathbb{C},\langle \Gamma , \rangle})_{\hat{W}}\rangle \cap \langle \nu _d(W)\rangle\in \mathbb {P}^N(\mathbb {R}),
\end{equation}
and $S_{\mathbb{C},W}$ evinces
$r_{\mathbb{C}} (O')$. 
If $(O',S_{\mathbb{C},W},{S_{\mathbb{R},W}})$ is either as in case (\ref{a}) or in case (\ref{c}) of the statement of the theorem, then $(O,S_{\mathbb{C}, \langle \Gamma \rangle},{S_{\mathbb{R}, \langle \Gamma \rangle}})$ is in the same case. 
Consider the system $\vert (1,0)\vert $ of lines on the smooth quadric surface $W$ containing $\Gamma$. We have that $h^1(W,\mathcal {O}_W(d-2,d))=0 $, and hence the restriction map $H^0(W,\mathcal {O}_W(d))\to
H^0(\Gamma,\mathcal {O}_{\Gamma}(d))$ is surjective. Therefore 
$$h^1(W,\mathcal {I}_{W\cap (S_{\mathbb{C}}\cup S_{\mathbb{R}})}(d)) = h^1(l,\mathcal {I}_{l\cap (S_{\mathbb{C}}\cup {S_{\mathbb{R}}}),l}(d))+h^1(r,\mathcal {I}_{r\cap (S_{\mathbb{C}}\cup {S_{\mathbb{R}}}),r}(d)).$$ 
Now, this last equality, together with the facts that $\nu _d(S_{\mathbb{C},W})$ and $\nu _d({S_{\mathbb{R},W}})$ are linearly independent and $(S_{\mathbb{C}}\cup {S_{\mathbb{R}}})_{W}
\subset \Gamma$, gives:
\begin{eqnarray}\label{eqa}
&\dim (\langle \nu _d(S_{\mathbb{C},W})\rangle \cap \langle \nu _d({S_{\mathbb{R},W}})\rangle) = 
\sharp (S_{\mathbb{C}}\cap S_{\mathbb{R}})_l
+ \sharp (S_{\mathbb{C}}\cap S_{\mathbb{R}})_r+&\nonumber  \\
&h^1(l,\mathcal {I}_{l\cap (S_{\mathbb{C}}\cup {S_{\mathbb{R}}}),l}(d))+h^1(r,\mathcal {I}_{r\cap (S_{\mathbb{C}}\cup {S_{\mathbb{R}}}),r}(d)).
\end{eqnarray}

\vspace{0.1cm}
\quad (1.2.1) Observe that (\ref{eqa}) implies that the case $\sharp (S_{\mathbb{C}}\cup {S_{\mathbb{R}}})_{r}= \sharp (S_{\mathbb{C}}\cup {S_{\mathbb{R}}})_{l} =d+1$ cannot happen because there is no contribution from $h^1(l,\mathcal {I}_{l\cap (S_{\mathbb{C}}\cup {S_{\mathbb{R}}}),l}(d))+h^1(r,\mathcal {I}_{r\cap (S_{\mathbb{C}}\cup {S_{\mathbb{R}}}),r}(d))$ since both terms, in this case, are equal to 0. So, we can assume that at least $\sharp (S_{\mathbb{C}}\cup {S_{\mathbb{R}}})_{l} >d+1$. 

\vspace{0.1cm}
\quad (1.2.2) Assume $\sharp (S_{\mathbb{C}}\cup {S_{\mathbb{R}}})_{r} =d+1$ and $\sharp (S_{\mathbb{C}}\cup {S_{\mathbb{R}}})_{l} >d+1$. 
\\
To prove that we are in case (\ref{a}) with respect to $l$ it is sufficient to prove $S_{\mathbb{C},r} = {S_{\mathbb{R},r}}$. 
\\
We have
$\dim \langle \nu _d(S_{\mathbb{C}}\cup {S_{\mathbb{R}}})\rangle = \sharp (S_{\mathbb{C}}\cup {S_{\mathbb{R}}})-1 -h^1(\mathcal {I}_{S_{\mathbb{C}}\cup {S_{\mathbb{R}}}}(d))$.
Since
$\nu _d(S_{\mathbb{C}})$ and $\nu _d({S_{\mathbb{R}}})$ are linearly independent, $\sharp (S_{\mathbb{C}}\cup {S_{\mathbb{R}}})=\sharp (S_{\mathbb{C}}) +\sharp ({S_{\mathbb{R}}})
-\sharp (S_{\mathbb{C}}\cap {S_{\mathbb{R}}})$,  $\sharp (S_{\mathbb{C},l}\cup {S_{\mathbb{R},l}})=\sharp (S_{\mathbb{C},l}) +\sharp (S_{\mathbb{R},l})-
\sharp (S_{\mathbb{C},l}\cap S_{\mathbb{R},l})$ and $h^1(\mathcal {I}_{S_{\mathbb{C}}\cup {S_{\mathbb{R}}}}(d))
= h^1(l,\mathcal {I}_{S_{\mathbb{C},l}\cup {S_{\mathbb{R},l}}})$, Grassmann's formula gives
that $\langle \nu _d(S_{\mathbb{C}})\rangle \cap \langle \nu _d({S_{\mathbb{R}}})\rangle$
is generated by $\langle \nu _d(S_{\mathbb{C},l})\rangle \cap \langle \nu _d(S_{\mathbb{R},l})\rangle$.
Since $P\notin \langle \nu _d(S')\rangle$ for any $S'\subsetneq S_{\mathbb{C}}$, we get
$S_{\mathbb{C}} = S_{\mathbb{C},l} \cup (S_{\mathbb{C}}\cap {S_{\mathbb{R}}})_{\hat{l}}$, i.e. $S_{\hat{l}} = {S_{\mathbb{R}}}\setminus S_{\mathbb{C},l}$. Hence
we can consider case (\ref{a}), and (\ref{SChat}) proves property (\ref{ai}) also for the case (1.2) that we are treating. The
point $P_l$ that we need to get (aii) can be identified with the point $O'$ defined in (\ref{Opset}) while (aiii) comes from
our hypotheses.

This gives all cases (\ref{a}) of Theorem \ref{i1}.

\vspace{0.1cm}
\quad (1.2.3) Assume that both $\sharp  (S_{\mathbb{C}}\cup {S_{\mathbb{R}}})_{r} \ge d+2$ and $\sharp (S_{\mathbb{C}}\cup {S_{\mathbb{R}}})_{l} \geq d+2$. 
\\
We need to prove that we are in case
(\ref{c}). 
Recall that $S_{\mathbb{C},{\hat{\Gamma}}} = S_{\mathbb{R},{\hat{\Gamma}}}$
and that $h^1(\mathcal {I}_{S_{\mathbb{C}, {\hat{\Gamma}}}\cup \Gamma}(d))=0$. The latter
equality implies, as in Remark \ref{d1}, that $\langle \{P\}\cup \nu _d(S_{\mathbb{C},{\Gamma}})\rangle
\cap \langle \nu _d(\Gamma)\rangle$ is a single real point $O_1$,
that $S_{\mathbb{C}, \Gamma}$ evinces $r_{\mathbb{C}}(O_1)$ and that ${S_{\mathbb{R},\Gamma}}$ evinces $O_1$. Now $O_1$ plays the role of $O_{\Gamma}$ of case (\ref{ciii}) in Theorem \ref{i1}.

Since
$\langle \nu _d(l)\rangle \cap \langle \nu_d(r)\rangle =\emptyset$ and $O_1\in \langle \nu _d(\Gamma)\rangle$,
the sets $\langle \{O_1\}\cup \nu _d(l)\rangle \cap \langle \nu _d(l)\rangle$ (resp. $\langle \{O_1\}\cup \nu _d(r)\rangle \cap
\langle \nu _d(r)\rangle$) are formed by a unique point $O_2$ (resp. $O_3$).
Remark \ref{d1} gives that $O_i\in \mathbb {P}^N(\mathbb {R})$, $i=1,2$, $S_{\mathbb{C},l}$ evinces $r_{\mathbb{C}} (O_2)$,
$S_{\mathbb{R},l}$ evinces $r_{\mathbb{R}} (O_2)$,  $S_{\mathbb{C},r}$ evinces $r_{\mathbb{C}} (O_3)$ and  ${S_{\mathbb{R},r}}$ evinces $r_{\mathbb{R}} (O_3)$.
The hypotheses of the case (1.2.3) coincide with (\ref{cii}) of the statement of the theorem, while (\ref{SChat}) gives
also property (\ref{ci}). Moreover $O_2$ and $O_3$ defined above coincide with $O_l$ and $O_r$ in (\ref{civ}) of Theorem \ref{i1}, therefore we have also proved case (\ref{c}) of Theorem \ref{i1}.

\begin{enumerate}
\setcounter{enumi}{\themycounter}

\item\label{2}  Now assume the existence of a conic $C\subset \mathbb {P}^m$ such that
\begin{equation}\label{CRC}
\deg (S_{\mathbb{C}}\cup {S_{\mathbb{R}}})_C \ge 2d+2.
\end{equation}

\end{enumerate}

Since $\sharp (S_{\mathbb{C}}\cup {S_{\mathbb{R}}})_C\le 3d-1-2d-2 \le d-1$, we
have $h^1(\mathcal {I}_{(S_{\mathbb{C}}\cup {S_{\mathbb{R}}})_{\hat{C}}}(d-2))=0$. By Lemma \ref{c2}
we have 
\begin{equation}\label{SCC}
S_{\mathbb{C}, \hat{C}} = {S_{\mathbb{R}, \hat{C}}},
\end{equation}
the set $\langle \{P\} \cup \nu_d(S_{\mathbb{C}, \hat{C}})\rangle  \cap
\langle \nu _d(\langle C \rangle)\rangle$ is a single point: 
\begin{equation}\label{Pp}
P':=\langle \{P\} \cup \nu_d(S_{\mathbb{C}, \hat{C}})\rangle  \cap
\langle \nu _d(\langle C \rangle)\rangle
\end{equation} 
and $S_{\mathbb{C},C}$ evinces
$r_{\mathbb{C}} (P')$. Moreover, if $C$ is defined over $\mathbb {R}$, then $P'\in \mathbb {P}^N(\mathbb {R})$ and $S_{\mathbb{R},C}$ evinces $r_{\mathbb{R}} (P')$. Hence $\sharp (S_{\mathbb{C},C})< \sharp (S_{\mathbb{C}}\cap {S_{\mathbb{R}}})$. 

\vspace{0.1cm}
\quad (2.1) Assume that $C$ is smooth.  Therefore (\ref{SCC}) proves (\ref{bi}) of the statement of the theorem in the case where $C$ is smooth. Since the reduced case is proved above (immediately after the displayed formula (\ref{SChat})), we have concluded the proof of (\ref{bi}).

Moreover the hypothesis (\ref{CRC}) coincides with (\ref{biii}) of the statement of the theorem since $\sharp (S_{\mathbb{C},C})$ is obviously strictly smaller than $ \sharp (S_{\mathbb{R},C})$. This concludes (\ref{biii}).

The fact that $\sharp (S_{\mathbb{C},C})< \sharp (S_{\mathbb{R},C})$ also implies that
$\sharp (S_{\mathbb{R},C}) \ge 5$. Since each point of ${S_{\mathbb{R}}}$ is real, $C$ is real. Remark \ref{d1} gives
that $ \nu_d({S_{\mathbb{R},C}})$ evinces $r_{\mathbb{R}} (P')$. Since ${S_{\mathbb{R},C}} \subset C$, ${S_{\mathbb{R},C}}$ also evinces
the real symmetric tensor rank of $P'$ with respect to the degree $2d$ rational normal
curve $\nu _d(C)$. 
The point $P'$ defined in (\ref{Pp}) plays the role of the point $P_C$ appearing in (\ref{bii}) of the statement of the theorem. Therefore, we have just proved  (\ref{bii}) of Theorem \ref{i1}.

We treat the case (2.2) below for sake of completeness, but we can observe that this concludes the proof of Theorem \ref{i1}.

\vspace{0.1cm}
\quad (2.2) Assume that $C$ is reducible, say $C =L_1\cup L_2$ with $L_1$ and $L_2$ lines
and $\sharp ((S_{\mathbb{C}}\cup {S_{\mathbb{R}}})_{L_1}) \ge \sharp ((S_{\mathbb{C}}\cap {S_{\mathbb{R}}})_{L_2})$.
 If $\sharp ((S_{\mathbb{C}}\cup {S_{\mathbb{R}}})\cap (L_2\setminus L_2\cap L_1)) \le d$, then we proved in step  (\ref{1}) that we are
in case (\ref{a}) with respect to the line $L_1$. Hence we may assume $\sharp ((S_{\mathbb{C}}\cup {S_{\mathbb{R}}})\cap (L_2\setminus L_2\cap L_1)) \ge d+1$. Thus even condition  (\ref{biv}) is satisfied as already remarked above after the displayed formula (\ref{SChat}).
\endproof

\vspace{0.3cm}

\def\fine{

We analyzed in step (a) all cases arising
if there is a line $L\subset C$ such that $\sharp (L\cap C) \ge d+2$. Hence we may assume
that there is no such line. In particular $C$ is not a double line. If $C$ is smooth, then $C$ is defined over $\mathbb {R}$, because $\sharp ( {S_{\mathbb{R},C}})\ge d+2 \ge 5$ and each point of ${S_{\mathbb{R}}}$ is real. If $C$ is reducible, say $C = C_1\cup C_2$, the Since $\sharp (S_{\mathbb{C}}\cup {S_{\mathbb{R}}})_C\le 3d-1-2d-2 \le d-1$, we
have $h^1(\mathcal {I}_{(S_{\mathbb{C}}\cup {S_{\mathbb{R}}})_{\hat{C}}}(d-2))=0$.

\begin{lemma}\label{c1}
Fix a line $l\subset \mathbb {P}^m$ defined
over $\mathbb {R}$, sets ${S_{\mathbb{C}}}_1\subset l(\mathbb {C})$, ${S_{\mathbb{R}}}_1\subset l(\mathbb {R})$,
$E\subset \mathbb {P}^m(\mathbb {R})\setminus l(\mathbb {R})$ and
$P_l\in \langle l\rangle (\mathbb {R})$ such that ${S_{\mathbb{C}}}_1\ne {S_{\mathbb{R}}}_1$, $\sharp (E) \le (d-1)/2$, ${S_{\mathbb{C}}}_1$ evinces $Crs (P_l)$
with respect to the rational normal curve $\nu _d(l)$ and ${S_{\mathbb{R}}}_1$ evinces $r_{\mathbb{R}} (P_l)$ with
respect to the rational normal curve curve $\nu _d(l)$ and that $\sharp ({S_{\mathbb{C}}}_1)<\sharp ({S_{\mathbb{R}}}_1)$. Fix
$P\in \langle \nu _d(E)\cup \{P_l\}\rangle (\mathbb {R})$ such that $P\notin \langle G\rangle$
for any $G\subsetneq \nu _d (E)\cup \{P_l\}$. Set $S_{\mathbb{C}}:= E\cup {S_{\mathbb{C}}}_1$ and ${S_{\mathbb{R}}}:= E\cup {S_{\mathbb{R}}}_1$.
Assume $\sharp (S_{\mathbb{C}})+\sharp ({S_{\mathbb{R}}}) \le 2d+2$ and $\sharp ({S_{\mathbb{C}}}_1)<\sharp ({S_{\mathbb{R}}}_1)$.

Then $S_{\mathbb{C}}$ evinces $r_{\mathbb{C}} (P)$ and ${S_{\mathbb{R}}}$ evinces $r_{\mathbb{R}} (P)$.
\end{lemma}

\begin{proof}
We first prove the lemma for $S_{\mathbb{C}}$. Assume that it is not true and take $F\subset \mathbb {P}^m(\mathbb {C})$ evincing $Crs (P)$. Since $\sharp (F) < \sharp (S_{\mathbb{C}}) \le d$, we
have $\sharp (S_{\mathbb{C}}\cup F) \le 2d-1$. Since $h^1(\mathcal {I}_{S_{\mathbb{C}}\cup F}(d)) >0$, we
get the existence of a line $r\subset \mathbb {P}^m$ such that $\sharp ((S_{\mathbb{C}}\cup F)) \ge d+2$.
Since $h^1(\mathcal {I}_{S_{\mathbb{C}}\cup F}(d)) >0$ (\cite{bb}, Lemma 1), there is a line $r\subset \mathbb {P}^m$
such that $\sharp (r\cap (S_{\mathbb{C}}\cup F)) \ge d+2$. Since $\sharp (S_{\mathbb{C}}\cup F\setminus (S_{\mathbb{C}}\cup F)\cap r))
\le d-1$, we have $h^1(\mathcal {I}_{S_{\mathbb{C}}\cup F\setminus (S_{\mathbb{C}}\cup F)\cap r)}(d-1))=0$.
Hence $F\setminus F\cap r = S_{\mathbb{C}}\setminus S_{\mathbb{C}}\cap r$. Since $\sharp (F)<\sharp (S_{\mathbb{C}})$, we
have $\sharp (F\cap r) < \sharp (S_{\mathbb{C}}\cap r)$. Hence $\sharp (S_{\mathbb{C}}\cap r) > (d+2)/2$. First assume $r\ne l$. Since
$\sharp (r\cap l) \le 1$, we have
$\sharp (E) \ge d/2$, a contradiction. Now assume $r=l$. Since $P\in \langle
\nu _d(F)\rangle$ and $P\notin \langle
\nu _d(F)\rangle$, the set $\langle \nu _d(F\setminus F\cap l)\cup \{P\}\langle \cap
\langle \nu _d(F\cap l)\rangle$ is a single point (call it $P_C)$. For
the same reason we have $\{P_l\} = \langle \nu _d(S_{\hat{l}})\cup \{P\}\langle \cap
\langle \nu _d(S_{\mathbb{C},l})\rangle$. Since $F\setminus F\cap l = S_{\hat{l}}$,
we have $P_C=P_l$. Since ${S_{\mathbb{C}}}_1$ evinces $r_{\mathbb{C}} (P_l)$, we get $\sharp (F\cap l)\ge \sharp
(S_{\mathbb{C}}\cap F)$. Hence $\sharp (F)\ge \sharp (S_{\mathbb{C}})$, a contradiction.

Now we prove the lemma for ${S_{\mathbb{R}}}$, using that it is true for $S_{\mathbb{C}}$.
Obviously ${S_{\mathbb{R}}}\subset \mathbb {P}^m(\mathbb {R})$ and $\sharp ({S_{\mathbb{R}}}) = \sharp ({S_{\mathbb{R}}}_1)+\sharp (E)$. We have
that $\sharp ({S_{\mathbb{R}}}_1)\le d$ (\cite{co}, ) and hence $\sharp ({S_{\mathbb{R}}}) \le (3d-1)/2$. By the real case of \cite{l}, Exercise 3.2.2.2, ${S_{\mathbb{R}}}_1$ evinces $r_{\mathbb{R}} (P_l)$. Assume that ${S_{\mathbb{R}}}$ does not
evinces $r_{\mathbb{R}} (P)$ and take $G\subset \mathbb {P}^m(\mathbb {R})$ such that $\sharp (G)< \sharp ({S_{\mathbb{R}}})$ and $G$ evinces $r_{\mathbb{R}} (P)$. First assume  $h^1(\mathcal {I}_{G\cup S_{\mathbb{C}}}(d)) >0$. Since
$\sharp (S_{\mathbb{C}})+\sharp (G) \le 2d$, there is a line $C_1$ such
that $\sharp (C_1\cap (S_{\mathbb{C}}\cup G)) \ge d+2$. Since
$\sharp (S_{\mathbb{C}}\cup G)-\sharp ((S_{\mathbb{C}}\cup G)\cap C_1)
\le 2d-d-2 \le d$, we have $h^1(\mathcal {I}_{(S_{\mathbb{C}}\cup G)\setminus (S_{\mathbb{C}}\cup G)\cap C_1)} (d-1)) =0$. Lemma \ref{c2} gives $G\setminus G\cap C_1 = S_{\mathbb{C}}\setminus {S_{\mathbb{C},C}}_1$. Hence
each point of $S_{\mathbb{C}}\setminus C_1$ is real. First assume
$C_1=l$. Hence $G\setminus G\cap l = E$. Since $S_{\mathbb{C}}$ evinces $r_{\mathbb{C}} (P)$, we
have $\sharp (G\cap l) \ge \sharp ({S_{\mathbb{C}}}_1)$ and in particular, $G\cap l \ne \emptyset$. Since
$G$ evinces $r_{\mathbb{R}} (P)$, the set $\langle \nu _d(G\cap l)\rangle \cap \langle \{P\} \cup E\rangle$
is a single point, $P_C$, and $G\cap l$ evinces $r_{\mathbb{R}} (P_C)$. Since
$h^1(\mathcal {I}_{(S_{\mathbb{C}}\cup G)\setminus (S_{\mathbb{C}}\cup G)\cap C_1)} (d-1)) =0$ and $C_1=l$, we have
$\langle \nu _d(l)\rangle \cap \langle \nu _d(E)\rangle =\emptyset$. Hence
$\langle \nu _d(l)\rangle \cap \langle \nu _d(E)\cup \{P\}\rangle$ is at most one point.
Since $P_l\in \langle \nu _d(l)\rangle \cap \langle \nu _d(E)\cup \{P\}\rangle$, we get
$P_l=P_C$. Hence $G\cap l$ evinces $r_{\mathbb{R}} (P_l)$. Hence $\sharp (G\cap l) =\sharp ({S_{\mathbb{R}}}_1)$.
Hence $\sharp (G)=\sharp ({S_{\mathbb{R}}})$, a contradiction.
Now assume $C_1\ne l$. Since $\sharp (C_1\cap l)\le 1$, and each point
of $E = S_{\hat{l}}$ is real, at most one point
of $S_{\mathbb{C}}$ is not real. Since $\sigma (S_{\mathbb{C}})=S_{\mathbb{C}}$, we get that
each point of $S_{\mathbb{C}}$ is real. Hence $Crs (P) =r_{\mathbb{R}} (P)$. Since
${S_{\mathbb{C}}}_1$ evinces $r_{\mathbb{C}} (P_l)$ and each point of $S_{\mathbb{C}}$ is real, we also get
$r_{\mathbb{R}} (P_l)=r_{\mathbb{C}} (P_l)$, i.e. $\sharp ({S_{\mathbb{C}}}_1)=\sharp ({S_{\mathbb{R}}}_1)$, i.e. $\sharp (S_{\mathbb{C}})=\sharp ({S_{\mathbb{R}}})$, a contradiction.
\end{proof}

}

\providecommand{\bysame}{\leavevmode\hbox to3em{\hrulefill}\thinspace}

\end{document}